\documentclass[oneside, 12pt]{amsart}
\usepackage{amscd, amsmath, verbatim, enumitem, graphicx, amssymb, mathtools}
\usepackage[all]{xy}
\usepackage{hyperref}
\hypersetup{
  pdfauthor={Sergey Sinchuk},
  pdftitle={On centrality of K2 for Chevalley groups of type E},  
  urlcolor=blue,
}
\usepackage[backend=bibtex, citestyle=numeric-comp, url=false, doi=false, eprint=true, maxbibnames=4]{biblatex}            

\renewbibmacro*{volume+number+eid}{\ifentrytype{article}{\- \iffieldundef{volume}{}{{\bf\printfield{volume}}}\iffieldundef{number}{}{ no.~\printfield{number},}}}
\renewbibmacro{in:}{\ifentrytype{article}{}{\printtext{\bibstring{in}\intitlepunct}}}
\newbibmacro{string+doi}[1]{\iffieldundef{doi}{\iffieldundef{url}{#1}{\href{\thefield{url}}{#1}}}{\href{http://dx.doi.org/\thefield{doi}}{#1}}}
\DeclareFieldFormat{title}{\usebibmacro{string+doi}{\mkbibemph{#1}}}
\DeclareFieldFormat[article, inproceedings, inbook, thesis]{title}{\usebibmacro{string+doi}{\mkbibquote{#1}}}

\numberwithin{equation}{section}
\newcounter{thmcounter} 
\newcounter{lemmacounter}
\newtheorem{thm}[thmcounter]{Theorem}
\newtheorem{prop}[thmcounter]{Proposition}
\newtheorem{cor}[thmcounter]{Corollary}
\newtheorem{lemma}[lemmacounter]{Lemma}
\theoremstyle{definition}
\newtheorem{rem}[equation]{Remark}

\newtheorem{dfn}[equation]{Definition}
\newtheorem{notation}[equation]{Notation}

\newcommand{\catname}[1]{{\normalfont\textbf{#1}}}
\newcommand{\Ker}{\operatorname{Ker}}
\newcommand{\Cent}{\operatorname{Cent}}
\newcommand{\Img}{\operatorname{Im}}

\newcommand{\coker}{\operatorname{coker}}
\newcommand{\K}{\operatorname{\mathrm{K}}}
\newcommand{\GG}{\operatorname{\mathrm{G}}}
\newcommand{\St}{\operatorname{\mathrm{St}}}
\newcommand{\StU}{\operatorname{\hat{\mathrm{U}}}}
\newcommand{\G}{\operatorname{\mathrm{G}}}
\newcommand{\E}{\operatorname{\mathrm{E}}}
\newcommand{\M}{\operatorname{\mathrm{M}}}
\newcommand{\sr}{\operatorname{\mathrm{sr}}}
\newcommand{\Um}{\operatorname{\mathrm{Um}}}
\newcommand{\rk}{\operatorname{\mathrm{rk}}}
\newcommand{\rA}{\mathsf{A}}

\newcommand{\rC}{\mathsf{C}} 
\newcommand{\rD}{\mathsf{D}} 
\newcommand{\rE}{\mathsf{E}}

\newcommand{\eval}[4]{ev \textstyle \left[\frac{#2[#1] \rightarrow #3}{#1 \mapsto #4}\right]}
\newcommand{\ev}[3]{\eval{t}{#1}{#2}{#3}}

\newcommand\restr[2]{{\left.\kern-\nulldelimiterspace #1 \vphantom{\big|}\right|_{#2}}}
  
\def\ssub#1{\mathchoice {_{\lower2pt\hbox{$\scriptstyle #1$}}} {_{\lower2pt\hbox{$\scriptstyle #1$}}}
     {_{\lower1.5pt\hbox{$\scriptscriptstyle #1$}}} {_{\!\lower1.5pt\hbox{$\scriptscriptstyle #1$}}}}

\title{On centrality of $\K_2$ for Chevalley groups of type $\rE_\ell$}

\author{S. Sinchuk}
\address{Chebyshev laboratory, St. Petersburg State University, St. Petersburg, Russia}
\email{sinchukss {\it at} yandex.ru}
\date {\today}

\begin{document}

\begin{abstract} For a root system $\Phi$ of type $\rE_\ell$ and arbitrary commutative ring $R$ we show that the group $\K_2(\Phi, R)$ is contained in the center of 
the Steinberg group $\St(\Phi, R)$. 
In course of the proof we also demonstrate an analogue of Quillen---Suslin local-global principle for $\K_2(\Phi, R)$. \end{abstract}

\maketitle

\section {Introduction}\label{intro}
Let $\Phi$ be a reduced irreducible root system and $R$ be a commutative ring with 1. 
Denote by $\G(\Phi, R)$ the simply connected Chevalley group of type $\Phi$ over $R$ and by $\E(\Phi, R)$ its elementary subgroup, i.\,e. the subgroup generated
by elementary root unipotents $t_\alpha(\xi)$, $\alpha\in \Phi$, $\xi\in R$, see~\cite{VP, Ta, St78}.
Taddei's theorem asserts that $\E(\Phi, R)$ is a normal subgroup of $\G(\Phi, R)$ provided $\rk(\Phi)\geq 2$ (see~\cite{Ta}).

Let $\St(\Phi, R)$ be the Steinberg group of type $\Phi$ over $R$ (see Definition~\ref{SteinbergDef}).
Denote by $\phi$ the canonical map $\St(\Phi, R)\rightarrow \GG(\Phi, R)$ sending each formal generator $x_\alpha(\xi)$ to the elementary root unipotent $t_\alpha(\xi)$.
The image of $\phi$ equals $\E(\Phi, R)$. The unstable $\K$-groups of M.~Stein $\K_1(\Phi, R)$ and $\K_2(\Phi, R)$ are defined as the cokernel and the kernel of $\phi$ (see~\cite[\S~1A]{St78}):
\begin{equation}\label{basicExSeq}\xymatrix{1 \ar[r] & \K_2(\Phi, R) \ar[r] & \St(\Phi, R) \ar[r]_{\phi} & \GG(\Phi, R) \ar[r] & \K_1(\Phi, R) \ar[r] & 1.}\end{equation}

It is classically known that $\K_2(\Phi, R)$ is contained in the center of $\St(\Phi, R)$ if $R$ is a local ring and $ \Phi$ has rank $\geq 2$ (see~\cite[Theorem~2.13]{St73}).
One of the standard conjectures in the theory of Chevalley groups over rings asserts that $\St(\Phi, R)$ is a {\it central} extension of $\E(\Phi, R)$ for \textit{any} commutative ring $R$ provided that the rank of $\Phi$ is sufficiently large.
We refer to this conjecture as {\it centrality of $\K_2$}. It can be regarded as a ``$\K_2$-analogue'' of Taddei's normality theorem.
For Chevalley groups of rank $2$ centrality of $\K_2$ fails already for $1$-dimensional $R$ (see~\cite[Theorem~1]{W}).

W.~van der Kallen in~\cite{Ka} and recently A.~Lavrenov in~\cite{La} proved centrality of $\K_2$ for $\Phi=\rA_\ell,\rC_\ell$, $\ell\geq 3$.
The main technical ingredient of both proofs is the so-called method of ``another presentation'' which consists in presenting $\St(\Phi, R)$
as a group with a set of generators modeling {\it elements of root type} (see~\cite[\S~4]{V00} for the idea of the definition).
For example, in the linear case the elements of root type are exactly the usual linear transvections, see Section~\ref{vdkP}.
The key advantage of this method lies in the fact that it allows one to define an action of $\G(\Phi, R)$ on $\St(\Phi, R)$ which turns the canonical map $\phi$ into a {\it crossed module}.

Recall that a {\it crossed module} is a triple $C=(N, M, \mu)$ consisting of an abstract group $N$ acting on itself by conjugation, an $N$-group $M$ 
and a group morphism $\mu\colon M \rightarrow N$ which preserves the action of $N$ and satisfies Peiffer identity $\mu (m) \cdot m' = m m' m^{-1}$, $m,m'\in M$.
The cokernel and the kernel of $\mu$ are usually denoted as $\pi_1(C)$ and $\pi_2(C)$, respectively.
It is not hard to show that the image of $\mu$ is a normal subgroup of $N$ (in particular, $\pi_1(C)$ is a group) and $\pi_2(C)$ is a central subgroup of $M$.
Moreover, the action of $N$ on $M$ induces an action of $\pi_1(C)$ on $\pi_2(C)$.
Thus, having constructed a crossed module $(\G(\Phi, R), \St(\Phi, R), \phi)$ (as was done e.\,g. in~\cite{Ka} for $\Phi=\rA_\ell$, $\ell\geq 3$)
we simultaneously obtain both centrality of $\K_2(\Phi, R)$ and normality theorem for $\E(\Phi, R)$.

The main disadvantage of the method of ``another presentation'' is its bulkiness and technical complexity. 
Moreover, this method requires separate analysis for each type of $\Phi$.
This is caused by the fact that the description of elements of root type depends on the ``global geometry'' of $\G(\Phi, R)$ (e.\,g. the structure of the minimal representation of $\G(\Phi, R)$).
Such description in the case $\Phi=\rE_\ell$, $\ell=6,7$ is particularly complicated already for fields and the case $\ell=8$ is even harder (see e.\,g.~\cite[\S~4]{V00} and further references therein).

The main goal of the present paper is to obtain a short unified proof of centrality of $\K_2$ for $\Phi=\rE_\ell$, $\ell=6,7,8$ which
involves only calculations with elementary root unipotents and bypasses the description of the set of relations between elements of root type.
More precisely, our main result is the following theorem.
\begin{thm}\label{centrality} Let $\Phi = \rE_\ell$, $\ell=6,7,8$. 
Then $\K_2(\Phi, R)$ is contained in the center of the Steinberg group $\St(\Phi, R)$. \end{thm}

The key ingredient in our proof of Theorem~\ref{centrality} is the following local-global principle.
\begin{thm}\label{QSlemma} Let $\Phi = \rE_\ell$, $\ell=6,7,8$.
An element $\alpha\in\St(\Phi, R[t], tR[t])$ is trivial if and only if its image in $\St(\Phi, R_M[t])$ is trivial for every maximal ideal $M\trianglelefteq R$. \end{thm}
In the above statement $\St(\Phi, R[t], tR[t])$ denotes the relative Steinberg group defined in section~\ref{relativeSteinberg}.
Notice that for $\Phi=\rA_\ell$, $\ell\geq 4$ a similar result was obtained by M.~Tulenbaev (cf.~\cite[Theorem~2.1]{T}).
In turn, Tulenbaev's theorem should be considered as a $\K_2$-analogue of D.~Quillen's local-global principle for projective modules \cite[Theorem~1]{Q76}
and A.~Suslin's local-global principle for $\K_1$ (see~\cite[Theorem~3.1]{Su}).

Within the context of theory of lower $\K$-functors modeled on Chevalley groups (see~\cite{St1}, \cite{St78}) centrality of $\K_2$ has a number of corollaries.
\begin{cor}\label{topoCorr} In the assumptions of Theorem~\ref{centrality} the following facts hold.
\begin{enumerate}[label=(\alph*)]
 \item \label{UCE} The group $\St(\Phi, R)$ is a universal central extension of $\E(\Phi, R)$ i.\,e. an initial object of the category of central extensions of $\E(\Phi, R)$. 
 \item \label{URCE} The relative Steinberg group $\St(\Phi, R, I)$ defined in section~\ref{relativeSteinberg} is a universal relative central extension of the morphism $\pi^*\colon \St(\Phi, R) \twoheadrightarrow \St(\Phi, R/I)$ in the sense of~\cite[\S~3]{Lo} (see definition below).
 \end{enumerate} \end{cor}
Let $\nu\colon N \twoheadrightarrow Q$ be a surjective map of abstract groups.
Recall that a {\it relative central extension} of $\nu$ is, by definition, a crossed module $C=(M, N, \mu)$ such that $\coker(\mu) = \nu$ and the induced action of $Q=\pi_1(C)$ on $\Ker(\mu)=\pi_2(C)$ is trivial.
A morphism of relative central extensions $(M, N, \mu) \rightarrow (M', N, \mu)$ of $\nu$ is a group morphism $\alpha\colon M\rightarrow M'$ preserving the action of $N$ and satisfying $\mu'\alpha=\mu$.
Now, a universal relative central extension of $\nu$ is, by definition, an initial object of the category of relative central extensions of $\nu$.

 \begin{cor}\label{topoCorr2} In the assumptions of Theorem~\ref{centrality} the following facts hold.
 \begin{enumerate}[label=(\alph*)]
 \item \label{RelSeq} If one defines $\K_2(\Phi, R, I)$ as the kernel of $\phi\colon\St(\Phi, R, I)\rightarrow\E(\Phi, R, I)$ then there is an exact sequence
 \begin{multline}\nonumber
     \mathrm{H}_3(\St(\Phi, R), \mathbb{Z}) \rightarrow \mathrm{H}_3(\St(\Phi, R/I), \mathbb{Z}) \rightarrow \K_2(\Phi, R, I) \rightarrow \K_2(\Phi, R) \rightarrow \K_2(\Phi, R/I) \rightarrow \\
     \rightarrow \K_1(\Phi, R, I) \rightarrow \K_1(\Phi, R) \rightarrow \K_1(\Phi, R/I).\end{multline}
 \item \label{Gersten} Let $X(R)$ be the result of application of Quillen's $+$-construction to the classifying space $B\GG(\Phi, R)$ with respect to the (perfect) subgroup $\E(\Phi, R)$.
                       Then the unstable groups $\K_i(\Phi, R)$ of M.~Stein coincide with $\pi_i(X(R))$, $i=1,2$ and the following unstable analogue of Gersten formula holds: $\pi_3(X(R))\cong \mathrm{H}_3(\St(\Phi, R), \mathbb{Z})$. \end{enumerate} \end{cor}

Tulenbaev's proof of~\cite[Theorem~2.1]{T} crucially depends on a certain property of the linear Steinberg group which is, in essence, a particular case of excision property (see~Remark~\ref{excision}).
In turn, our proof of Theorem~\ref{QSlemma} goes as follows.
First, we show that the relative Steinberg group of type $\rE_\ell$ can be presented as amalagamated product of several copies of Steinberg groups of type $\rA_4$.
Next, we use this presentation to extend the specified property from pieces of type $\rA_4$ to the whole group $\St(\rE_\ell, R, I)$.

The rest of the article is organized as follows. 
In section~\ref{prelim} we introduce principal notation and prove certain facts about root systems.
In section~\ref{relativeSteinberg} we define unstable relative Steinberg groups following the approach of F.~Keune and J.-L.~Loday.
We also state some results comparing this definition with the one used by M.~Tulenbaev in~\cite{T}.
Finally, in section~\ref{mainReduction} we prove our main results.
 
\subsection{Acknowledgements}
The author is grateful to V.~Petrov and S.~O.~Ivanov who suggested two key ideas upon which this article is based.
The author would also like to thank E.~Kulikova, A.~Luzgarev, A.~Lavrenov, A.~Smolensky, V.~Sosnilo, A.~Stavrova, N.~Vavilov and the referee for their numerous helpful comments which helped to improve this article.
The author acknowledges financial support from Russian Science Foundation grant 14-11-00297.

\section {Preliminaries}\label{prelim}
Throughout this paper all commutators are left-normed i.\,e. $[x, y]=xy x^{-1}y^{-1}$ and all rings are unital and commutative unless otherwise stated.

\subsection{Root-theoretic lemmata}
In what follows $\Phi$ denotes a reduced irreducible root system and $\Pi\subseteq \Phi$ denotes its basis (i.e. the set of simple roots).
Denote by $\widetilde{\alpha}$, $\Phi^+$ and $\Phi^-$, respectively, the maximal root of $\Phi$ and the subsets of positive and negative roots of $\Phi$.
The Dynkin diagram and the extended Dynkin diagram of $\Phi$ corresponding to $\Pi$ will be denoted by $D(\Phi)$, $\widetilde{D}(\Phi)$, respectively.

A proper closed root subset $S\subseteq \Phi$ is called {\it parabolic} (resp. {\it reductive}, resp. {\it special}) if $\Phi=S \cup -S$ (resp. $S = -S$, resp. $S \cap -S=\varnothing$).
Any parabolic subset $S \subseteq \Phi$ can be decomposed into the disjoint union of its reductive and special part, i.e. $S = \Sigma_S \sqcup \Delta_S$, where $\Sigma_S \cap (-\Sigma_S) = \varnothing$, $\Delta_S = -\Delta_S$.
 
Denote by $m_\beta(\alpha)$ the coefficient of $\beta$ in the expansion of $\alpha$ in $\Pi$, i.\,e. $\alpha = \sum_{\beta\in\Pi} m_\beta(\alpha) \beta$.
For $\beta\in\Pi$ denote by $\Delta_\beta$ the subsystem of $\Phi$ spanned by all simple roots except $\beta$ and by $\Sigma_\beta$ the set consisting of roots $\alpha\in \Phi$ such that $m_\beta(\alpha)>0$.

We denote by $(\alpha, \beta)$ the scalar product of roots and by $\langle \beta, \alpha\rangle$ the integer $2(\beta, \alpha)/(\alpha, \alpha)$.
The Weyl group $W(\Phi)$ is a subgroup of isometries of $\Phi$ generated by all reflections $\sigma_\alpha$, where $\sigma_\alpha(\beta)=\beta-\langle\beta,\alpha \rangle\cdot \alpha$.
For a subset of roots $S \subseteq \Phi$ we denote by $\langle S \rangle$ the root subsystem spanned by $S$, i.e. the minimal subset of $\Phi$ containing $S$ and invariant under the action of reflections $\sigma_\alpha$, $\alpha\in S$.

For a simply laced $\Phi$ (i.e. $\Phi$ such that $\widetilde{D}(\Phi)$ does not contain multiple bonds) denote by $\Phi'$ the subsystem of $\Phi$ consisting of all roots orthogonal to $\widetilde{\alpha}$.
From a consideration of $\widetilde{D}(\Phi)$ it follows that $\Phi'$ has the following type (depending on the type of $\Phi$).
\[\begin{tabular}{| l | l | l | l | l | l | l |} \hline
    Type of $\Phi$  & $\rA_\ell,\ \ell\geq 3$     & $\rD_\ell,\ \ell\geq 4$           & $\rE_6$ & $\rE_7$ & $\rE_8$ \\ \hline
    Type of $\Phi'$ & $\rA_{\ell-2}$ & $\rA_1+\rD_{\ell-2}$ & $\rA_5$ & $\rD_6$ & $\rE_7$ \\ \hline \end{tabular}\]
Notice that in the above table we identify $\rD_2$ with $\rA_1+\rA_1$ and $\rD_3$ with $\rA_3$.
    
\begin{lemma}\label{rpLemma} Let $\Phi$ be a simply laced irreducible root system of rank $\geq 3$.
Every two roots $\alpha, \beta \in \Phi$ are contained in some subsystem $\Psi\subseteq\Phi$ of type $\rA_3$. \end{lemma}
\begin{proof}
If $\alpha$, $\beta$ are not orthogonal then they are contained in some subsystem of type $\rA_2$, denote it by $\Psi$.
From the irreducibility of $\Phi$ it follows that the set $D = \Phi\setminus(\Psi \cup \Psi^\bot )$ is nonempty.
Now for any $\gamma \in D$ the subsystem $\langle \Psi \cup \{\gamma\} \rangle$ is the required one.

Now assume that $\alpha\perp\beta$. We can choose a basis of $\Phi$ so that $\alpha=-\widetilde{\alpha}$. 
Denote by $\gamma$ any of the simple roots adjacent to $\alpha$ on $\widetilde{D}(\Phi)$. 
 
If $\Phi=\rD_\ell$, $\ell\geq 4$ and $\beta$ is contained in a component of $\Phi'$ of type $\rA_1$, then $\Psi := \langle\{ \alpha, \gamma, \beta \} \rangle$ satisfies the requirement of the lemma.
If $\rk\Phi\geq 5$ denote by $\Phi''$ either $\Phi'$ itself (if $\Phi\neq \rD_\ell$) or the irreducible component $\rD_{\ell-2}$ of $\Phi'$ (if $\Phi=\rD_\ell$). 
In the remaining case we may assume that $\beta$ lies in $\Phi''$. Denote by $\delta\in\Phi''$ the simple root adjacent to $\gamma$.
The Weyl group $W(\Phi'')$ fixes $\alpha$ and acts transitively on $\Phi''$.
Therefore, there exists $w\in W(\Phi'')$ such that $w(\delta) = \beta$ and $\Psi := w(\langle \{ \alpha, \gamma, \delta \} \rangle)$ is the desired subsystem. 
\end{proof}

Denote by $A_n(\Phi)$ the set of all subsystems $\Psi\subseteq\Phi$ of type $\rA_n$. The group $W(\Phi)$ naturally acts on $A_n(\Phi)$.
If $\alpha$ is a root of $\Phi$ denote by $A_n(\Phi, \alpha)$ the subset consisting of subsystems $\Psi\in A_n(\Phi)$ containing $\alpha$.

\begin{lemma}\label{a2pc} Let $\Phi$ be an irreducible simply laced root system of rank $\geq 3$.
Let $\Psi_0, \Psi_1\subseteq \Phi$ be two different subsystems of type $\rA_3$ containing a common root $\alpha$.
Then either $\Psi_0 \cap \Psi_1$ has type $\rA_2$, or there exists a subsystem $\Psi \in A_3(\Phi, \alpha)$ such that $\Psi_0 \cap \Psi$, $\Psi_1 \cap \Psi$ both have type $\rA_2$. \end{lemma}
\begin{proof} As in the proof of the previous lemma we may assume that $\alpha=-\widetilde{\alpha}$. 
Denote by $\gamma$ any simple root adjacent to $\alpha$ on $\widetilde{D}(\Phi)$ (it is unique if $\Phi\neq \rA_\ell$, otherwise there is only one other simple root, denote it by $\gamma'$).

For any $\beta_1, \beta_2 \in \Sigma_\gamma$ the subsystem $\Psi=\langle \{\alpha, \beta_1, \beta_2 \}\rangle$ has type either $\rA_2$ or $\rA_3$.
Moreover, any $\Psi\in A_3(\Phi)$ containing $\alpha$ can be obtained in this fashion. 
Indeed, if $\Phi\neq \rA_\ell$ this follows from the fact that for any $\beta\in\Phi^+\setminus\{-\alpha\}$ not orthogonal to $\alpha$ one has $m_\gamma(\beta) = \frac{(\beta, \alpha)}{(\gamma, \alpha)} = 1$.
In the case $\Phi=\rA_\ell$ it is possible that $m_{\gamma'}(\beta) = 1$, $m_\gamma(\beta) = 0$.
In this case one should additionally use the fact that $\beta' = \widetilde{\alpha} - \beta \in \Phi^+$ and $m_{\gamma}(\beta') = 1$.

Now let $\Psi_0 = \langle \{ \alpha, \beta_1, \beta_2 \} \rangle$, $\Psi_1 = \langle \{ \alpha, \beta_3, \beta_4 \} \rangle$ be two subsystems of type $\rA_3$.
First, consider the case when $\Psi_0 \cap \Psi_1 = \langle \{ \alpha \} \rangle$.
Clearly, $\Psi=\langle \{ \alpha, \beta_1, \beta_3 \} \rangle$ has type $\rA_3$ while both intersections $\Psi\cap \Psi_0$, $\Psi \cap \Psi_1$ have type $\rA_2$.
In the case when $\Psi_0 \cap \Psi_1$ has type $\rA_1 + \rA_1$ we may assume without loss of generality that $\Psi_0 \cap \Psi_1 = \langle \{ \alpha, \gamma \}\rangle $
for $\alpha \perp \gamma$ and $\Psi_0=\langle \{ \alpha, \beta_1, \gamma \}\rangle$, $\Psi_1 = \langle\{ \alpha, \beta_2, \gamma \} \rangle$.
In this case the subsystem $\Psi = \langle \{ \alpha, \beta_1, \beta_2 \} \rangle$ satisfies the requirements of the lemma. \end{proof}

\begin{lemma}\label{ext34} For $\Phi=\rE_\ell$, $\ell=6,7,8$ the following statements hold:
\begin{enumerate}
 \item \label{ExtProp} Any subsystem $\Psi \subseteq \Phi$ of type $\rA_3$ is contained in some $\Psi'\subseteq \Phi$ of type $\rA_4$;
 \item \label{a2pc4}   For any $\Psi_0, \Psi_1 \in A_4(\Phi, \alpha)$ either $\Psi_0\cap \Psi_1$ contains a subsystem from $A_2(\Phi, \alpha)$
       or there exists $\Psi \in A_4(\Phi, \alpha)$ such that both $\Psi\cap\Psi_0$ and $\Psi\cap\Psi_1$ contain a subsystem from $A_2(\Phi, \alpha)$.
\end{enumerate} \end{lemma}
\begin{proof} For $\Phi$ of type $\rE_\ell$ the action of $W(\Phi)$ on $A_3(\Phi)$ is transitive (see discussion after~\cite[Theorem~5.4]{Dyn52}).
 To prove~\ref{ExtProp} it suffices to find on $\widetilde{D}(\Phi)$ a subdiagram of type $\rA_3$ which is contained in a subdiagram of type $\rA_4$.
 The second statement of the lemma is a formal consequence of the first statement and Lemma~\ref{a2pc}. \end{proof}        

\begin{rem} An analogue of the first statement of the previous lemma is false in the case $\Phi=\rD_\ell$ even for $\ell\geq 5$. 
The reason for this lies in the fact there are two orbits in $A_3(\rD_\ell)$ under the action of $W(\rD_\ell)$, $\ell\geq 5$.
In Tables 4--5 of~\cite{Ca72} these orbits are labeled as $\rA_3$ and $\rD_3$. 
A subsystem contained in the orbit $\rA_3$ can be embedded into a subsystem of type $\rA_4$, while this is not true for a subsystem lying in the orbit $\rD_3$. \end{rem}
 
\subsection{Steinberg groups} \begin{dfn}\label{SteinbergDef}
If $\rk(\Phi)\geq 2$ the {\it Steinberg group} $\St(\Phi, R)$ can be defined as the group given by generators $x_{\alpha}(\xi)$, $\xi\in R$, $\alpha\in\Phi$ and the following set of relations.
\begin{equation}\label{Radd} x_\alpha(\xi) x_\alpha(\eta) = x_\alpha(\xi+\eta),\ \alpha \in \Phi, \xi, \eta \in R;\end{equation} 
\begin{equation}\label{Rcf} [x_\alpha(\xi),  x_\beta(\eta)] = \prod\limits x_{i\alpha + j\beta}(N_{\alpha,\beta, i, j}\xi^i \eta^j),\quad \alpha\neq-\beta,\ \alpha,\beta\in \Phi,\ \xi, \eta \in R.\end{equation}
\end{dfn}
The indices $i$, $j$ appearing in the right-hand side of the above formula range over all positive natural numbers such that $i\alpha + j\beta\in\Phi$.
The integers $N_{\alpha, \beta, i, j}$ are called {\it structure constants} of the Chevalley group $\GG(\Phi,R)$ and depend only on $\Phi$.
More information on Chevalley groups and their structure constants can be found in~\cite[\S~9]{VP}.

Let $I\leq R$ be an additive subgroup and $\alpha\in \Phi$ be a root. Denote by $X_\alpha(I)$ the additive subgroup of $\St(\Phi, R)$ consisting of $x_\alpha(s)$, $s\in I$. 
We call such subgroups $X_\alpha(I)$ {\it root subgroups}.

Generators $x_\alpha(\xi)$ should be thought of as formal symbols modeling elementary root unipotents $t_\alpha(\xi)\in\E(\Phi, R)\leq \GG(\Phi, R)$,
while relations~\ref{Radd}--\ref{Rcf} are {\it elementary} relations between $t_\alpha(\xi)$ which hold over arbitary $R$.

Throughout this paper we will be mainly concerned with root systems whose Dynkin diagram is simply laced.
In this case formula~\ref{Rcf} further simplifies to
\begin{equation}\label{Rcf21} [x_\alpha(\xi),  x_\beta(\eta)] = 1,\ \alpha+\beta\notin\Phi\cup\{0\},\end{equation}
\begin{equation}\label{Rcf22} [x_\alpha(\xi),  x_\beta(\eta)] = x_{\alpha+\beta}(N_{\alpha\beta}\xi\eta),\ \alpha+\beta\in\Phi,\ \text{where}\ N_{\alpha\beta}=\pm 1. \end{equation}
The integers $N_{\alpha\beta}$ mentioned above are precisely the structure constants of the simple complex Lie algebra $L_\mathbb{C}(\Phi)$ of type $\Phi$ with respect to some positive Chevalley base $\{e_\alpha\}$ (see~\cite[\S~1]{V}, \cite[\S~2]{VP}),
i.e. $[e_\alpha, e_\beta]=N_{\alpha\beta}e_{\alpha+\beta}$, $\alpha,\beta\in\Phi$.

\begin{rem} For $\xi\in R^*$ set $w_\alpha(\xi) := x_\alpha(\xi) x_{-\alpha}(-\xi^{-1}) x_{\alpha}(\xi).$
If the rank of $\Phi$ is at least 2 then the following identity is a consequence of \ref{Radd}--\ref{Rcf}:
\begin{equation}\label{RW} w_\alpha(\xi) x_{\beta}(\zeta) w_\alpha(\xi)^{-1} =
x_{\sigma\ssub{\alpha}\beta} (\eta_{\alpha, \beta}\xi^{-\langle\beta,\alpha \rangle} \zeta),\ \zeta\in R,\ \xi\in R^*. \end{equation}
In the above formula the integers $\eta_{\alpha, \beta}$ are equal $\pm 1$ and can be deduced from the structure constants (see~\cite[\S~13]{VP}). \end{rem}
\begin{rem}
The definition~\ref{SteinbergDef} is not suitable in the case $\Phi=\rA_1$ since the set of relations of type~\ref{Rcf} becomes empty.
There exist different definitions of the Steinberg group of rank $1$.
For example, R.~Steinberg and M.~Stein defined $\St(\rA_1, R)$ as a group given by relations \ref{Radd}, \ref{RW} (cf. with~\cite[Definition~3.7]{St1}).
Some authors further enforce all the standard relations for elements $h_\alpha$, $w_\alpha$ (see e.g. relations 1--7 of~\cite{A}).
Throughout this article we stick to the definition of $\St(\rA_1, R)$ of M.~Stein and R.~Steinberg.
However, as we will see below, most of our statements do not depend on the choice of definition for $\St(\rA_1, R)$. \end{rem}

We now turn to functoriality properties of Steinberg groups.
Clearly, there exists a well-defined Steinberg group functor $\St(\Phi, -)\colon\catname{Rings} \rightarrow \catname{Grp}$.
For a map of rings $f\colon R\rightarrow S$ the induced map $\St(\Phi, f)$ sending each generator $x_\alpha(\xi)$ to $x_\alpha(f(\xi))$ will be denoted by $f^*$.

\begin{notation}\label{Znot} Let $I\trianglelefteq R$ be an ideal of $R$. For $s\in I$, $\xi\in R$ set $Z_\alpha(s, \xi) := x_\alpha(s)^{x_{-\alpha}(\xi)} \in \St(\Phi, R)$.
For notational convenience, in the sequel we continue to denote the elements of an ideal $I$ by Latin letters (e.g. $s, t \in I$) and elements of the ambient ring $R$ by Greek letters (e.g. $\xi, \eta \in R$).\end{notation}

\begin{lemma}\label{titsLemma} Let $\Phi$ be an irreducible root system of rank $\geq 2$. 
Let $G$ denote the normal closure in $\St(\Phi, R)$ of the subgroup spanned by elementary root unipotents of level $I$, i.\,e. $G=\langle x_\alpha(s) \mid \alpha\in\Phi,\ s\in I\rangle^{\St(\Phi, R)}$.
\begin{enumerate}
 \item \label{Tits1} The group $G$ is generated by elements $Z_{\alpha}(s,\xi)$, where $s\in I$, $\xi\in R$, $\alpha\in\Phi$.
 \item \label{Tits2} Let $S$ be a parabolic set of roots of $\Phi$ with special part $\Sigma_S$. Then $G$ is generated by the following two families of elements:  
       $x_{\alpha}(s)$, where $s\in I$, $\alpha\in\Phi$ and $Z_\alpha(s,\xi)$, where $s\in I$, $\xi\in R$, $\alpha\in\Sigma_S$.
\end{enumerate} \end{lemma}
\begin{proof} 
If one replaces $x_\alpha(s)$ by $t_\alpha(s)$ and $\St(\Phi, R)$ by $\E(\Phi, R)$ the first statement of the lemma becomes well-known.
Indeed, the group $\E(\Phi, R, I):=\langle t_\alpha(s) \mid \alpha\in\Phi,\ s\in I\rangle^{\E(\Phi, R)}$ is generated by $t_{\alpha}(s)^{t_{-\alpha}(\xi)}$, $s\in I$, $\xi\in R$ as proved, for example, in~\cite[Theorem~2]{Va}, ~\cite[Proposition~3.2]{S}.
Both these proofs rely solely on calculations with Steinberg relations~\ref{Radd}--\ref{Rcf} and hence can be reproduced verbatim for Steinberg groups. 
The stronger second statement is an analogue of~\cite[Theorem~3.4]{S}. \end{proof}
\begin{lemma}\label{titsLemma2} 
Let $\Phi$ be an irreducible root system of rank $\geq 2$ and let $\pi\colon R\rightarrow R/I$ denote the canonical projection.
\begin{enumerate} 
 \item The group $G$ coincides with $\Ker(\pi^*\colon \St(\Phi, R) \rightarrow \St(\Phi, R/I))$.
       In particular $\Ker(\pi^*)$ admits either of the two generating sets described in Lemma~\ref{titsLemma}.
 \item If $\pi$ admits a section $f \colon R/I\rightarrow R$ then $G$ is generated by elements of the form $Z_\alpha(s, \xi)$, $\xi \in f(R/I)$, $s\in I$, $\alpha \in \Phi$.
\end{enumerate}
\end{lemma}
\begin{proof}
 The first statement can be proved similarly to~\cite[Lemma~6.1]{M71}, cf. also \cite[(3.12)]{St1}.
 
 Let $g$ be an element of $G$. By the first part of Lemma~\ref{titsLemma} we can present $G$ as a finite product of generators of the form $Z_\alpha(s, \xi)$, $s \in I$, $\xi \in R$.
 Now every such generator can be factored as $Z_\alpha(s, \xi) = Z_{-\alpha}(-s', 0) \cdot Z_\alpha(s, f\pi(\xi)) \cdot Z_{-\alpha}(s', 0),$ where $s' = \xi - f\pi(\xi) \in I.$
\end{proof}

Let $\Phi$ be a simply laced root system. Denote by $\catname{Subsys}(\Phi)$ the category of root subsystems of $\Phi$.
Its objects are all root subsystems $\Psi\subseteq \Phi$ and its morphisms are in one-to-one correspondence with inclusions of the form $\Psi_1\subseteq \Psi_2$.

\begin{lemma}\label{funct} For a given simply laced root system $\Phi$ we can define the Steinberg group functorially with respect to taking subsystems of $\Phi$, i.\,e. we can define functor
$$\St(-, -)\colon \catname{Subsys}(\Phi)\times \catname{Rings} \rightarrow\catname{Grp}.$$ \end{lemma}
\begin{proof}
Fix signs $N_{-,-}\colon \Phi^2\rightarrow\{\pm 1\}$ of the structure constants for $\Phi$.
Now, one can obtain structure constants for any subsystem $\Psi\subseteq \Phi$ by restricting $N_{-,-}$ on $\Psi^2\subseteq\Phi^2$.
Now for any embedding $i\colon \Psi_1 \hookrightarrow \Psi_2$ of root subsystems of $\Phi$ the corresponding morphism $\St(i, R)$ can be defined by the obvious formula:
$\St(i, R)(x_\alpha(\xi))=x_\alpha(\xi)$, $\xi \in R$, $\alpha\in \Psi_1$. \end{proof}

\begin{dfn}\label{Gdef}
Denote by $\mathcal{G}_{1,n}$ the full subcategory of $\catname{Subsys}(\Phi)$ with $\mathrm{Ob}(\mathcal{G}_{1,n}) := A_1(\Phi)\sqcup A_n(\Phi)$. 
Clearly, $\mathcal{G}_{1,n}$ is a directed bipartite graph whose only nonidentical morphisms are inclusions of the form $i_{\gamma, \Psi}\colon \langle \{ \gamma \} \rangle \hookrightarrow \Psi$, $\gamma\in\Psi$, $\Psi\in A_n(\Phi)$. \end{dfn}

Recall the definition of amalgamated products of groups (which are pushouts in the category of abstract groups).
Given maps $\varphi\colon F\rightarrow G$ and $\psi\colon F \rightarrow H$ the amalgamated product $G *_{F} H$ is defined as the quotient of $G * H$ modulo the smallest normal subgroup generated by $\{ \psi(g) \varphi(g)^{-1} \mid g\in F \}$.
Notice that for a surjective map of groups $\pi\colon\widetilde{F}\rightarrow F$ the pushout $G *_{\widetilde{F}} H$ taken with respect to $\varphi\pi$ and $\psi\pi$ will be isomorphic to $G *_F H$.

Although the next lemma will not be used directly in the sequel, it illustrates an important idea that Steinberg groups can be presented as amalgamated products of Steinberg groups of smaller rank.
\begin{lemma}\label{absPres} Let $\Phi$ be a simply laced irreducible root system of rank $\geq 3$.
Then $\St(\Phi, R)$ is isomorphic to $G := \varinjlim_{\mathcal{G}_{1,3}}\St(-,R).$ \end{lemma}
\begin{proof}
Clearly, $G$ can be interpreted as the free product of $|A_3(\Phi)|$ copies of $\St(\rA_3, R)$ modulo relations which identify the images of 
all generators $x_{\pm\alpha}(\xi)$, $\alpha\in \Phi$ with respect to all maps $\St(i_{\alpha,\Psi}, R)$, $\Psi \in A_3(\Phi, \alpha)$.
Notice that the last assertion does not depend on the exact definition of $\St(\rA_1, R)$ 
(in fact, we only need to know that our definition of $\St(\rA_1, R)$ imposes few enough relations between $x_\alpha(\xi)$ so that the maps $\St(i_{\alpha, \Psi}, R)$ are well-defined
which is the case for both definitions mentioned in Remark~\ref{SteinbergDef}). 

The above argument shows that $G$ and $\St(\Phi, R)$ have the same set of generators. It is clear that the defining relations of $G$ hold in $\St(\Phi, R)$. 
It suffices to show the converse, namely that every defining relation of $\St(\Phi, R)$ occurs in the set of defining relations of $\St(\Psi, R)$ for some $\Psi\in A_3(\Phi)$.
But this is a consequence of the form of relations~\ref{Rcf21}--\ref{Rcf22} and Lemma~\ref{rpLemma}. \end{proof}

\begin{rem} A much stronger variant of Lemma~\ref{absPres} called {\it Curtis---Tits presentation} is known for Steinberg groups (see~\cite[Corollary~1.3]{A}).
This presentation is formulated in terms of subgroups corresponding to vertices and edges of $D(\Phi)$ (rather than subgroups corresponding to subsystems of $\Phi$).
Moreover, it remains valid in the broader context of Kac--Moody groups, i.e. it holds for Steinberg-like groups obtained from generalized Cartan matrices. \end{rem}

\subsection{Van der Kallen's ``another presentation'' of the linear Steinberg group}\label{vdkP}
Recall that a column $v=(v_1,\ldots,v_n)^t\in R^{n}$ is called {\it unimodular} if $R v_1+\ldots+ R v_n=R.$
More generally, a matrix $v \in M_{n\times m}(R)$ is called unimodular if there exists $u\in M_{m\times n}(R)$ such that $uv$ is the identity matrix of size $m$.
We denote the set of unimodular columns by $\Um(n, R)$.

Let $U$ denote the set consisting of pairs $(v, u)$, where $v$ is a unimodular column of height $n$ and $u$ is a row of length $n$ such that $uv=0$. 
Notice that for any $(v, u)\in I$ the matrix $T_{v, u} = e + vu$ is invertible and elementary with inverse $e - vu$. 
Such matrices are called (linear) {\it transvections}.

Let $n\geq 4$. Consider the group $\St(n, R)$ defined by generators $X_{v,u}$ where $(v, u)\in U$ and relations
\begin{equation}\label{AP1} X_{v,u_1+u_2} = X_{v,u_1} \cdot X_{v, u_2}; \end{equation}
\begin{equation}\label{AP2} {X_{v,u}}^{X_{v',u'}} = X_{(e - v'u')v, u (e + v'u')}.\end{equation}
The following relation follows from~\ref{AP1}--\ref{AP2} (see discussion preceding Lemma~1.1 of~\cite{T}):
\begin{equation}\label{AP3} X_{v_1+bv_2,u} = X_{v_1,u}\cdot X_{v_2,bu},\end{equation}
if $b\in R$, $(v_1,v_2)\in\M_{n\times2}(R)$ is a unimodular matrix and $u\in {}^n\!R$ is a row such that $uv_1 = uv_2 = 0$.

Choose any basis $\{e_i\}$, $1\leq i \leq \ell+1$ of the free module of columns $R^{\ell+1}$. 
For a column $v \in R^{\ell+1}$ denote by $v^t\in{}^{\ell+1}\! R$ its transpose.
The following theorem is the main result of~\cite{Ka}.
\begin{thm}\label{vdkTheorem} For $l\geq 3$ the map sending each generator $x_{ij}(\xi)$ to $X_{e_i, \xi e_j^t}$, $1\leq i \neq j\leq \ell+1$ defines an isomorphism of groups $\St(\rA_\ell, R)$ and $\St(\ell+1, R)$. \end{thm}
\begin{proof} See~\cite[Theorem~1]{Ka}. \end{proof}

\begin{rem} Below we continue to denote by $\St(n, -)$ linear Steinberg groups in ``another presentation''.
At the same time we retain notation $\St(\rA_\ell, -)$ for groups in the ``usual presentation''. \end{rem}

There is some degree of freedom in the choice of the set of generators for $\St(n, R)$.
For example, we could replace the requirement $v\in \Um(n, R)$ in the definition of $U$ by $v\in \E(n, R) \cdot e_1$ and get the same group $\St(n, R)$ (this follows e.\,g. from Proposition~\ref{T16} below).
Notice that in general $\E(n, R)\cdot e_1$ is strictly smaller than $\Um(n, R)$ (and they are equal only under some assumptions on the dimension of $R$, e.\,g. $\sr(R)\leq n-1$, where $\sr(R)$ denotes the stable rank of $R$).
We also note that the presentation of the linear Steinberg group formulated in Theorem~\ref{vdkTheorem} is not the only possible one.
There is a variant of another presentation which does not involve unimodular columns at all (cf. \cite[Theorem~2]{Ka}).
                        
\section {Relative Steinberg groups}\label{relativeSteinberg}
Recall that in the context of~\cite{St1} an ordered pair $(R, I)$ consisting of a ring $R$ and an ideal $I\trianglelefteq R$ is called simply a {\it pair}.
A {\it morphism of pairs} $\varphi\colon (A, I)\rightarrow (A', I')$ is a ring morphism $\varphi\colon A\rightarrow A'$ satisfying $\varphi(I)\subseteq I'$.
We denote the category of pairs by \catname{Prs}.

Let $R$ be a ring and $I\trianglelefteq R$ be an ideal.
Denote by $D(R, I)$ the {\it double} of $R$ along $I$, i.e. the pullback of two copies of $R$ over $R/I$.
\begin{equation}\label{doubleRing} \xymatrix{\ar@{}[dr] D(R, I) \ar[r]^-{p_1} \ar[d]_{p_2} & R \ar[d]^\pi \\ R \ar[r]_-\pi & R/I.} \end{equation}
The elements of $D(R, I)$ can be interpreted as ordered pairs $a=(a_1,  a_2)\in R^2$ such that $a_1-a_2\in I$.
Maps $p_i$ are defined by $p_i(a)=a_i$, $i=1,2$.
Both $p_1$ and $p_2$ are split by the diagonal map $\Delta\colon R \rightarrow D(R, I)$. Clearly, $D$ is a functor from $\catname{Prs}$ to $\catname{Rings}$.

\begin{dfn}\label{semidirectProd}
Let $R$ be a ring and $A$ be a (not necessarily unital) $R$-algebra. 
Consider the product $R\times A$ as an abelian group with respect to componentwise addition.
Define the ring structure on $R\times A$ using the formula $(a; b) (c; d) = (ac; ad + bc + bd)$.
The resulting unital ring is denoted by $R\ltimes A$. Clearly, $0\times A$ is an ideal in $R\ltimes A$. \end{dfn}

It is easy to see that $R \cong R/I\ltimes I$ whenever the quotient map $R\twoheadrightarrow R/I$ admits section.
Applying this consideration to the exact sequence $0\times I\hookrightarrow D(R, I)\xrightarrow{p_1} R$ we get that $D(R, I) \cong R\ltimes I$ for any $R$, $I$.
To distinguish between these two equivalent interpretations of elements of $D(R, I)$ we adopt the following convention:
we write $(a_1, a_2)_I$ whenever we consider $D(R, I)$ as the fibered product $R \times_{R/I} R$ and use notation $(a; s)$ for elements of $D(R, I)$ if we interpret the latter as the semidirect product $R \ltimes I$.
In particular, $ (a_1, a_2)_I = (a_1; a_2 - a_1),\ (a; s) = (a, a+s)_I,\ a, a_1, a_2 \in R,\ s\in I.$

\subsection{The definition of the relative Steinberg group}\label{relativeSteinbergDef}
In the current section we define {\it relative} Steinberg groups and study their basic properties.
We follow the approach used by F.~Keune and J.-L.~Loday in the stable situation (cf.~\cite{Ke}, \cite{Lo}).

Consider the following two subgroups of $\St(\Phi, D(R, I))$: 
$$G_i := \Ker\left(p_i^*\colon \St(\Phi, D(R, I)) \rightarrow \St(\Phi, R)\right).$$ 
The groups $G_1$ and $G_2$ coincide with the normal closure in $\St(\Phi, D(R, I))$ of the subgroups spanned by
root unipotents $x_\alpha((0, s)_I)$, $x_\alpha((s, 0)_I)$, $s\in I$, respectively (see Lemma~\ref{titsLemma2}).

Denote by $C$ the mixed commutator subgroup $[G_1, G_2]\leq \St(\Phi, D(R, I)).$
Clearly, $C$ is a normal subgroup of $\St(\Phi, D(R, I))$ contained in $G_1\cap G_2$.
For every $g\in G_1\cap G_2$ we have that $p_i^*(\phi(g))=1_{\G(\Phi, R)}$, $i=1,2$ hence $G_1\cap G_2$ is contained in $\K_2(\Phi, D(R, I))$.
\begin{dfn}\label{relSt} Define the {\it relative} Steinberg group $\St(\Phi, R, I)$ as the quotient $G_1/C$. \end{dfn}
It is easy to see that the relative Steinberg group is a central extension of the normal subgroup of $\St(\Phi, R)$ spanned by elements $x_\alpha(s)$, $s\in I$.
More precisely, there is an exact sequence
\begin{equation}\label{suite}\begin{xymatrix} {1 \ar[r] & (G_1\cap G_2)/C \ar[r] & \St(\Phi, R, I) \ar[r]^{\overline{p_2^*}} & \St(\Phi, R) \ar[r]^{\pi^*} & \St(\Phi, R/I) \ar[r] & 1.} \end{xymatrix}\end{equation}
The definition of $\St(\Phi, R, I)$ is functorial in both $(R, I)$ and $\Phi$, i.e. there is a functor $\St\colon\catname{Subsys}(\Phi)\times \catname{Prs}\rightarrow\catname{Grp}$ (cf. Lemma~\ref{funct}).

In one important case the relative Steinberg group will be a subgroup of the absolute group $\St(\Phi, R)$.
\begin{lemma}\label{relSubgr} Let $\Phi$ be a root system of rank $\geq 2$ and $I\trianglelefteq R$ be an ideal such that the canonical map $R\rightarrow R/I$ splits.
Then the map $\overline{p_2^*}$ from sequence~\ref{suite} is injective. \end{lemma}
\begin{proof}
 Using the isomorphism $R \cong R/I \ltimes I$ we get that $$D(R, I) = R\times_{R/I}R \cong (R/I\ltimes I)\times_{R/I} (R/I\ltimes I).$$
 Now it is easy to see that $D(R, I)$ is isomorphic to $R/I\ltimes(I\times I)$ with the isomorphism defined by 
 $((\overline{\xi}; s_1), (\overline{\xi}; s_2))_I \mapsto (\overline{\xi}; (s_1, s_2))$, $\overline{\xi}\in R/I$, $s_1, s_2 \in I$.

 Consider two maps $\theta_1, \theta_2\colon R \rightarrow D(R, I)$ sending $(\overline{\xi}; s)\in R/I \times I$ to $(\overline{\xi}; (s, 0))$ and $(\overline{\xi}; (0, s))$, respectively. 
 Let $g$ be an element of $G_1\cap G_2$. Clearly, $g$ lies in the kernel of $(\pi p_1)^* = (\pi p_2)^*$ and hence by Lemma~\ref{titsLemma2}
 it can be presented as a finite product $\prod_{i=1}^n g^{i}$ of generators of the form $g^i=Z_\alpha((0; (s_1, s_2)), (\overline{\xi}; (0, 0)))$, $s_1, s_2\in I$, $\overline{\xi} \in R/I$.
 
 Now each generator $g^i$ can be factored as $g^i_1 g^i_2$ where
 $$g^i_1=(\theta_1 p_1)^*(g^i) = Z_{\alpha}\left((0; (s_1, 0)), (\overline{\xi}; (0, 0))\right)\ \text{and}\ g^i_2=(\theta_2 p_2)^*(g^i)=Z_{\alpha}\left((0; (0, s_2)); (\overline{\xi}; (0, 0))\right)$$
 lying in $G_2$ and $G_1$ respectively.
 Since $g^i_1$ and $g^j_2$ commute modulo $C$ for every $1\leq i,j\leq n$, we can rearrange factors in the decomposition of $g$ so that in the quotient $\St(\Phi, D(R, I))/C$ we have
 $$\overline{g} =  \prod_{i=1}^n \overline{g^i_1} \overline{g^i_2} = \prod_{i=1}^n \overline{g^i_1} \cdot \prod_{i=1}^n \overline{g^i_2} = \overline{(\theta_1 p_1)^*(g)} \cdot \overline{(\theta_2 p_2)^*(g)} = \overline{\theta_1^*(1)} \cdot \overline{\theta_2^*(1)} = 1. \qedhere$$
\end{proof}

From now on we will use simpler notation for certain elements of $\St(\Phi, R, I)$.
For $s\in I$, $\xi\in R$ set
$$z_\alpha(s, \xi) := Z_{\alpha}((0, s)_I, \Delta(\xi))C = x_\alpha((0, s)_I)^{x_{-\alpha}(\Delta(\xi))}C,\ \ y_\alpha(s) := z_\alpha(s, 0).$$
Applying Lemma~\ref{titsLemma2} to the projection $p_1$ and section $\Delta$ and passing to the quotient modulo $C$ we obtain that elements $z_\alpha(s, \xi)$ generate $\St(\Phi, R, I)$ as an abstract group if $\rk(\Phi)\geq 2$.

\begin{lemma}\label{Zrels} Let $\Phi$ be a simply laced root system and $\alpha,\beta\in \Phi$ be such that $\alpha+\beta\in\Phi$.
Then elements $z_\alpha(s, \xi)$ satisfy the following relations ($\xi, \eta\in R$, $s\in I$):
\begin{enumerate} 
\item\label{Z1} $z_{\alpha}(s, \xi) ^ {x_{-\alpha}(\eta)} = z_{\alpha}(s, \xi + \eta)$;
\item\label{Z2} $z_{\alpha}(s, \xi) ^ {x_{\beta}(\eta)} = y_{\beta} (- s\xi \eta) \cdot y_{\alpha+\beta} (N_{\alpha,\beta}\cdot s\eta)     \cdot z_{\alpha}(s, \xi)$;
\item\label{Z3} $z_{\alpha}(s, \xi) ^ {x_{-\beta}(\eta)} = y_{-\beta} (s\xi \eta) \cdot y_{-\alpha-\beta} (N_{\alpha,\beta}\cdot s\xi^2\eta) \cdot z_{\alpha}(s, \xi)$;
\item\label{Z4} $z_{\alpha}(s, \xi) ^ {x_{\gamma}(\eta)} = z_{\alpha}(s, \xi)\ \text{if}\ \gamma\perp\alpha$.
\end{enumerate} \end{lemma}
\begin{proof} These formulae can be checked directly using relations~\ref{Rcf21}--\ref{Rcf22} and identities for structure constants (see~\cite[\S~1]{V} or \cite[\S~14]{VP}). \end{proof}

Let $X$ be a set and $G$ be a group. Denote by $F$ the group freely generated by $X\times G$. Let $R$ be a subset of $F$.
\begin{dfn} A group $H$ is said to be {\it presented as a $G$-group} by generators $X$ and relations $R$ if it is isomorphic to the quotient of $F$ modulo the smallest normal $G$-invariant subgroup containing $R$. \end{dfn}
It is easy to see that $G$ naturally acts on $H$ on the right.
Below we use more familiar notation $x^g$ for each generator $(x, g)\in X\times G$.

Now, we are ready to formulate an analogue of Swan's presentation for relative Steinberg groups (cf. with~\cite[\S~4]{Lo}, \cite[Proposition~11]{Ke}). 
\begin{prop}\label{SwanP}
 Let $\Phi$ be a simply laced irreducible root system of rank $\geq 2$.
 Then $\St(\Phi, R, I)$ can be presented as a $\St(\Phi, R)$-group with generators $y_{\alpha}(s)$, $s\in I$, $\alpha\in\Phi$ and relations:
 \begin{enumerate}
  \item \label{R1} $y_{\alpha}(s_1)y_{\alpha}(s_2) = y_{\alpha}(s_1+s_2),\ s_1,s_2\in I;$
  \item \label{R2} $[y_{\alpha}(s_1), y_{\beta}(s_2)] = 1,\ \alpha+\beta\not\in\Phi;$
  \item \label{R3} $[y_{\alpha}(s_1), y_{\beta}(s_2)] =  y_{\alpha+\beta}(N_{\alpha,\beta}\cdot s_1 s_2),\ \alpha+\beta\in\Phi;$
  \item \label{R4} $y_{\alpha}(s)^{x_{\beta}(\xi)} = y_{\alpha}(s),\ \xi\in R,\ \alpha+\beta\not\in\Phi,\ \alpha\neq -\beta$;
  \item \label{R5} $y_{\alpha}(s)^{x_{\beta}(\xi)} = y_{\alpha}(s) y_{\alpha+\beta}(N_{\alpha,\beta}\cdot \xi s),\ \alpha+\beta\in\Phi;$  
  \item \label{R6} $y_{\alpha}(s) ^ {g \cdot x_{\beta}(t)} = y_{\beta}(-t) \cdot y_{\alpha}(s)^g \cdot y_{\beta}(t),\ s, t\in I, g\in \St(\Phi, R),\alpha,\beta\in\Phi. $
 \end{enumerate}
\end{prop}
\begin{proof}
Denote by $H$ the $\St(\Phi, R)$-group given by generators $y_\alpha(s)$ and relations \eqref{R1}--\eqref{R5}.
One can repeat the argument of R.~Swan (cf.~\cite[Lemma~7.8]{Sw71}, \cite[Lemma~8.4]{Sw70}) and show that $G_1$ and $H$ are isomorphic.
One has to verify that the map $\theta$ which sends $x_\alpha((\xi_1, \xi_2)_I) \in \St(\Phi, D(R, I))$ to $(y_\alpha(\xi_2-\xi_1), x_\alpha(\xi_1)) \in H \rtimes \St(\Phi, R)$ defines an isomorphism of these groups.
By restricting $\theta$ on $G_1$ we obtain the needed isomorphism of $G_1$ and $H$. 

It remains to verify that $G_1/C$ is isomorphic to the quotient of $H$ modulo the smallest $\St(\Phi, R)$-invariant normal subgroup generated by all words corresponding to relation~\eqref{R6}, i.e. words of the form
$w=y_\alpha(s)^g \cdot y_\beta(t) \cdot y_{\alpha}(-s)^{g \cdot x_\beta(t)} \cdot y_\beta(-t)$. 
It is easy to check that $w = \theta(c)$, where $c = [x_\alpha((0, s)_I)^{\Delta^*(g)}, x_\beta((-t, 0)_I)] \in C$.
It remains to notice that using standard commutator identities any element of $C=[G_1, G_2]$ can be rewritten as a product of conjugates of commutators of this form. 
\end{proof}

\begin{rem}\label{R62} From Lemma~\ref{titsLemma2} it follows that one can rewrite $y_{\alpha}(s_1)^g$ modulo relations~\eqref{R1}--\eqref{R5} as a product of elements of the form $y_{\alpha'}(s) ^ {x_{-\alpha'}(\xi)}$.
Therefore, one can additionally assume in the statement of relation~\eqref{R6} that $g = x_{-\alpha}(\xi)$, $\xi\in R$. \end{rem}

One can consider Proposition~\ref{SwanP} as a universal property of relative Steinberg groups.
Indeed, it allows one to construct a unique map $\St(\Phi, R, I) \rightarrow G$ whenever one chooses certain elements $y'_\alpha(s)\in G$ and defines an action of $\St(\Phi, R)$ on $G$ which behaves on $y'_\alpha(s)$ according to relations \eqref{R1}--\eqref{R6}.

\subsection{Case $\Phi = \rA_\ell$, $\ell\geq 3$.}
The main goal of this subsection is to show that the abstract definition of the relative Steinberg group in the linear case agrees with 
Tulenbaev's definition formulated in terms of another presentation (cf.~\cite[Definition~1.5]{T}).

\begin{dfn}\label{TulDef}
Let $U_{I, R}$ denote the set consisting of pairs $(v, w)$ where $v$ is a column lying in the orbit $\E(n, R)\cdot e_1$ and $w \in {}^n\! I$ is a row satisfying $wv=0$.
For $n\geq 4$ the relative linear Steinberg group $\St(n, R, I)$ is the group defined by generators $X_{v,w}$,  $(v, w)\in U_{I, R}$ and relations \ref{AP1}--\ref{AP3}. \end{dfn} 

There is a well-defined action of $\St(n, R)$ on $\St(n, R, I)$ defined by the formula ${X_{v,w}}^g = X_{\phi(g^{-1})\cdot v, w \cdot \phi(g)}$, where $g\in \St(n, R)$ and $\phi$ is the canonical map from~\ref{basicExSeq}.
From the definition of $\St(n, R, I)$ it also follows that there is a map $j\colon \St(n, R, I) \rightarrow \St(n, R)$ induced by the inclusion $U_{I, R}\subseteq U$.
\begin{prop} \label{T16} Let $I$ be an ideal such that the canonical projection $\pi\colon R\rightarrow R/I$ splits, then the following facts hold:
\begin{enumerate}
 \item The map $j$ is injective and $\St(n, R, I)$ can be interpreted as a subgroup of $\St(n, R)$;
  \item $\St(n, R)$ is isomorphic to $\St(n, R/I) \ltimes \St(n, R, I)$ (the action $\St(n, R/I)$ on $\St(n, R, I)$ is induced by the conjugation action of $\St(\Phi, R)$ via the splitting map).
\end{enumerate} \end{prop}
\begin{proof} See~\cite[Proposition~1.6]{T}. \end{proof}

In the rest of this subsection we identify vectors $v\in D(R, I)^n$ with pairs $(v, \widetilde{v}) \in R^n \times R^n$ such that $v-\widetilde{v}\in I^n$.
If a column $(v,\widetilde{v})\in D(R, I)^n$ is unimodular then both vectors $v, \widetilde{v} \in R^n$ are unimodular.

Applying Proposition~\ref{T16} to the map $p_1$ from diagram~\ref{doubleRing} we get that $$G_1=\Ker(p_1^*) \cong \St(n, D(R, I), 0\times I).$$
In particular, $G_1$ is generated as an abstract group by elements $X_{(v, \widetilde{v}), (0, \widetilde{w})}$, $((v, \widetilde{v}), (0, \widetilde{w}))\in U_{0\times I, D(R, I)}$.
Similar statement holds for $G_2$.

Now it is easy to see that $C\leq \Ker(p_1^*)$ is generated by elements of the form 
\begin{equation}\label{KLAP0} [X_{(v_1, \widetilde{v_1}), (0, \widetilde{w_1})}, X_{(v_2, \widetilde{v_2}), (w_2, 0)}] = 
   X_{(v_1, \widetilde{v_1}), (0, \widetilde{w_1})} X_{((e + v_2 w_2)v_1, \widetilde{v_1}), (0, -\widetilde{w_1})}.\end{equation}
Recall that $\E(n, R, I)$ is generated as an abstract group by transvections $T_{v_2, w_2} = e + v_2w_2$, $(v_2, w_2)\in U_{I, R}$ (indeed, the conjugate of any root unipotent $t_\alpha(s)$, $s\in I$ has this form).
The last statement together with~\ref{KLAP0} imply that
\begin{equation}\label{KLAP} X_{(v_1, \widetilde{v_1}), (0, \widetilde{w_1})}C = X_{(gv_1, \widetilde{v_1}), (0, \widetilde{w_1})}C\ \text{for any}\ g\in\E(n, R, I).\end{equation}

\begin{prop} Let $n\geq 4$, then Tulenbaev's relative group $\St(n, R, I)$ is isomorphic to Loday's relative group $\St(\rA_{n-1}, R, I)$. \end{prop}
\begin{proof}
Define the map $\psi' \colon \Ker(p_1^*) \rightarrow \St(n, R, I)$ by 
$$\psi'(X_{(v, \widetilde{v}), (0, \widetilde{w})}) = X_{\widetilde{v}, \widetilde{w}},\ ((v, \widetilde{v}), (0, \widetilde{w}))\in U_{0\times I, D(R, I)}.$$
It is clear that $\psi'$ preserves the defining relations~\ref{AP1}--\ref{AP3} of $\Ker(p_1^*)$ while from relations~\ref{AP1} and~\ref{KLAP0} it follows that $\psi'(C)=1$.
Therefore, there exists an induced map $\psi\colon \St(\rA_{n-1}, R, I)\rightarrow \St(n, R, I)$.
   
Now, let us construct a map inverse to $\psi$.
Define $\varphi\colon\St(n, R, I)\rightarrow\St(\rA_{n-1}, R, I)$ using the formula $\varphi(X_{v,w})=X_{(v,v), (0,w)}C.$
It is clear that $\varphi$ preserves relations~\ref{AP1} and~\ref{AP3}. 
From identity~\ref{KLAP} it follows that relation~\ref{AP2} is also preserved by $\varphi$:
\begin{multline}\nonumber \varphi({X_{v,w}}^{X_{v',w'}}) = {X_{\Delta(v),(0,w)}}^{X_{\Delta(v'),(0,w')}}C = X_{(v, v - (v'w')v), (0,w + w(v'w'))}C = \\
= X_{\Delta((e - v'w')v), (0, w + w(v'w'))}C = \varphi(X_{(e - v'w')v, w (e + v'w')}). \end{multline}
Let us show that $\varphi$ and $\psi$ are inverse to each other.
Only equality $\varphi\psi = id$ is nontrivial. 
Let $((v, \widetilde{v}), (0, \widetilde{w}))$ be an element of $U_{0\times I, D(R, I)}$ and let $g=(g_1, g_2)\in \E(n, D(R, I))$ 
be such that $(g_1, g_2) (v,\widetilde{v}) = (e_1, e_1)$. 
From $\pi^*(g_2^{-1} \cdot g_1) = \pi^*(p_2^*(g)^{-1} \cdot p_1^*(g)) = 1$ we get that $g_2^{-1}\cdot g_1 \in \E(n, R, I)$ and from~\ref{KLAP} we obtain that
$$X_{(v, \widetilde{v}), (0, \widetilde{w_1})}C = 
X_{((g_2^{-1}\cdot g_1)v, \widetilde{v}), (0, \widetilde{w_1})}C = 
X_{(\widetilde{v}, \widetilde{v}), (0, \widetilde{w_1})}C,\ \text{as required}. \qedhere$$
\end{proof}

\begin{rem} Recently, A.~Lavrenov obtained a variant of another presentation for relative {\it symplectic} Steinberg groups (see~\cite{La2}). \end{rem}

\subsection{Tulenbaev's map}
Let $a \in R$ be a nonnilpotent element.
Denote by $\lambda_a\colon R\rightarrow R_a$ the morphism of principal localization at $a$ (i.e. localization at $\{1, a, a^2, a^3, \ldots\}$).
Similarly, if $M\trianglelefteq R$ is a prime ideal we denote by $\lambda_M$ the morphism of localization at $R\setminus M$.

Let $R_a[t]$ be a polynomial algebra over $R_a$. Clearly, $tR_a[t]$ is an algebra over $R$. This allows us to form semidirect product $R \ltimes tR_a[t]$ (see~Definition~\ref{semidirectProd}).
Elements of this ring can be interpreted as polynomials whose free terms belong to $R$ and all other coefficients lie in $R_a$.
There is a map $\theta\colon R[t]\rightarrow R\ltimes tR_a[t]$ which localizes all coefficients of terms of degree $\geq 1$ at $a$.

The key ingredient in Tulenbaev's proof of the local-global principle for linear Steinberg groups is the following observation.
\begin{lemma}\label{Tmap} For $n\geq 5$ the restriction of the map $\theta^*$ to the subgroup $\St(n, R[t], tR[t])$
factors through $\St(n, R_a[t], tR_a[t])$, i.e. there exists a map $T$ such that the following diagram commutes:
\begin{equation}\label{diagT}\begin{xymatrix} {                     & \St(n, R[t], tR[t]) \ar[dl]_{\lambda_{a}^*} \ar[dr]^{\theta^*} & \\
                  \St(n, R_a[t],tR_a[t]) \ar@{-->}^T[rr] &                     & \St(n, R\ltimes tR_a[t]).}  \end{xymatrix}\end{equation}\end{lemma}
\begin{proof} See~\cite[Lemma~2.3]{T}. \end{proof}

\begin{rem}\label{excision} It is not hard to show that the image of $T$ lies inside $\St(n, R\ltimes tR_a[t], tR_a[t])$ interpreted as a subgroup of $\St(n, R\ltimes tR_a[t])$ by virtue of Lemma~\ref{relSubgr}.
Moreover, it can be shown that $T$ defines an isomorphism of $\St(n, R_a[t], tR_a[t])$ and $\St(n, R\ltimes tR_a[t], tR_a[t])$ with the inverse map induced by $\lambda_a\colon R\ltimes tR_a[t]\rightarrow R_a[t]$.
This observation allows one to interpret Lemma~\ref{Tmap} as a special case of excision property for Steinberg groups.\end{rem}

\subsection{Relative Steinberg groups as amalgams}
Below we always assume that $\Phi$ is an irreducible simply laced root system of rank $\geq 3$.

The main purpose of the current section is to demonstrate the following relative analogue of Lemma~\ref{absPres}.
\begin{thm}\label{relPres} The group $\St(\Phi, R, I)$ is isomorphic to $G:=\varinjlim_{\mathcal{G}_{1,3}}\St(-, R, I)$ as an abstract group. \end{thm}
Morally speaking, the above result should be considered as an indirect description of the set of relations between generators $z_\alpha(s, \xi)$ of the relative group $\St(\Phi, R, I)$.
The direct description of the set of relations between $z_\alpha(s, \xi)$ (defining $\St(\Phi, R, I)$ as an abstract group, not as a $\St(\Phi, R)$-group) is currently unknown.
Theorem~\ref{relPres} asserts that such set can be obtained by joining the sets of relations for subsystems $\St(\Psi, R, I)$, $\Psi\in A_3(\Phi)$.

The proof Theorem~\ref{relPres} is somewhat trickier than that of Lemma~\ref{absPres}.
The reason for this is that $\St(\Phi, R, I)$ comes with the natural action of $\St(\Phi, R)$ while a priori there is no such action on $G$.
Our immediate goal is to construct this action explicitly. The proof of Theorem~\ref{relPres} occupies the rest of the section.

Denote by $j_\Psi$ the canonical map $\St(\Psi, R, I)\rightarrow G$.
There is a map $i:G \rightarrow \St(\Phi, R, I)$ induced by $i_\Psi\colon \St(\Psi, R, I)\rightarrow \St(\Phi, R, I)$, i.\,e. $i_\Psi=i \cdot j_\Psi$. 
To help distinguish between elements of different factors of the amalagamated product $G$ we write $z_\alpha^{\Psi}(s, \xi)$ whenever we want to emphasize the fact that $z_\alpha^{\Psi}(s, \xi)$ lies in $\St(\Psi, R, I)$
for a particular $\Psi\in A_3(\Phi, \alpha)$.
Notice that by the definition of $G$ for $\alpha\in \Phi$, $\Psi_0, \Psi_1\in A_3(\Phi, \alpha)$ we have $j_{\Psi_0}(z_\alpha^{\Psi_0}(s, \xi)) = j_{\Psi_1}(z_\alpha^{\Psi_1}(s, \xi))$, $s\in I$, $\xi\in R$.

\begin{lemma} Let $\Psi\subseteq \Phi$ be a subsystem of type $\rA_3$, let $\alpha\in\Phi$ be a root.
For $\eta\in R$ there exists a map $c^\Psi_\alpha(\eta)\colon \St(\Psi, R, I)\rightarrow G$ modeling conjugation with $x_\alpha(\eta)$ in $\St(\Phi, R, I)$,
i.e. such that the following diagram commutes:
\[\begin{xymatrix}{\St(\Psi, R, I) \ar[r]^{c^\Psi_\alpha(\eta)} \ar[d]_{i_\Psi} & G \ar[d]^{i} \\
\St(\Phi, R, I) \ar[r]^{(-)^{x_\alpha(\eta)}} & \St(\Phi, R, I).} \end{xymatrix}\]
\end{lemma}
\begin{proof}
The cases when $\alpha\perp\Psi$ or $\alpha\in\Psi$ are obvious, indeed, set
\begin{equation}\label{defc1}
 c^{\Psi}_\alpha(\eta) =
  \begin{cases} 
       j_\Psi    \hfill & \text{ if } \alpha\perp\Psi, \\
       j_\Psi \circ (-)^{x_{\alpha}(\eta)} \hfill & \text{ if } \alpha\in\Psi. \\
  \end{cases} \end{equation}

Now consider the case $\alpha\not\perp\Psi$, $\alpha\not\in\Psi$. Denote by $S$ the closed subset of roots generated by $\alpha$ and $\Psi$.
Let $\Psi'$ be the minimal root subsystem of $\Phi$ containing $S$. Clearly, $S$ is a parabolic subset of $\Psi'$.
One can write $S = \Psi\sqcup\Sigma$, where $\Sigma\ni\alpha$ is the special part of $S$.
It is easy to see that $\Psi'$ has type $\rA_4$ or $\rD_4$ and that $\Sigma$ consists of $4$ or $6$ roots, respectively.
 
Denote by $\StU(\Sigma, I)$ the subgroup of $G$ spanned by all $y_\gamma(s)$, $\gamma\in\Sigma$, $s\in I$.
It is not contained in a single factor of rank $3$ however it is easy to show that $\StU(\Sigma, I)$ is an abelian group isomorphic to $V_I := I^{|\Sigma|}$.
Indeed, by Lemma~\ref{rpLemma} for every $\gamma_1,\gamma_2\in\Sigma$ the root subgroups $X_{\gamma_1}(I)$, $X_{\gamma_2}(I)$ are contained in 
$j_{\Psi''}(\St(\Psi'', R, I))$ for some $\Psi''\in A_3(\Phi)$ and hence commute by Lemma~\ref{Zrels}.
Denote by $\psi$ the isomorphism $V_I \xrightarrow{\cong} \StU(\Sigma, I)$.
 
Set $H := \E(\Psi, R)=\langle t_\alpha(\xi) \mid \alpha \in \Psi, \xi\in R \rangle \leq \GG(\Phi, R)$ and denote by $U$ the unipotent radical subgroup of $\GG(S, R)$, i.e. the subgroup spanned by $t_{\alpha}(\xi)$, $\alpha\in\Sigma$, $\xi\in R$.
Denote by $V$ the free $R$-module $R^{|\Sigma|}$ isomorphic to $U$ and by $e_\alpha$ the basis vector of $V$ corresponding to the root subgroup $X_\alpha(R) \leq U$ under this isomorphism.
The conjugation action of $H$ on $U$ yields a representation $\rho\colon H\rightarrow \mathrm{Aut}_R(V)$. 
By the definition of $\rho$ for $h\in H$ one has
\begin{equation}\nonumber h\cdot u = \rho(h)(u)\cdot h,\ \text{for}\ u\in V_I\leq V.\end{equation}
 
The next step is to show that an analogue of the above relation holds in $G$. 
Set $\rho' := \rho\phi$, where $\phi$ denotes the canonical map $\St(\Psi,R, I)\rightarrow \E(\Psi, R, I)\leq H$.
We argue that for $g\in \St(\Psi, R, I)$ one has
\begin{equation}\label{reprRels} \begin{tabular}{c} $j_\Psi(g)\cdot \psi(u) = \psi(\rho'(g)(u))\cdot j_\Psi(g),\ \text{for}\ u\in V_I\leq V.$\end{tabular}\end{equation}
Indeed, it suffices to check these relations only for generators of $\St(\Psi, R, I)$ of the form $g=z_\beta(a,\xi)$, $u = s\cdot e_\gamma$, $s\in I$, $\gamma\in\Sigma$.
Under these assumptions relation~\ref{reprRels} has the same form as relations \ref{Z2}--\ref{Z3} of Lemma~\ref{Zrels} and by Lemma~\ref{rpLemma} it can be found among the relations of some $\St(\Psi'', R, I)$, $\Psi''\in A_3(\Phi)$.  
Now we are ready to define the desired map $c^\Psi_\alpha(\eta)$. For $g\in\St(\Psi, R, I)$ set
\begin{equation}\label{defc2} c^\Psi_\alpha(\eta)(g) := \psi(\rho'(g)(\eta e_\alpha) - \eta e_\alpha) \cdot j_\Psi(g) \in G. \end{equation}
This formula agrees with relations~\eqref{Z2}--\eqref{Z3} from Lemma~\ref{Zrels}.
Moreover, from~\ref{reprRels} it follows that $c^\Psi_\alpha(\eta)$ is a group homomorphism:
\begin{multline}\nonumber c^\Psi_\alpha(\eta)(g_1)\cdot c^\Psi_\alpha(\eta)(g_2) = \psi(\rho'(g_1)(\eta e_\alpha) - \eta e_\alpha) \cdot j_\Psi(g_1) \cdot \psi(\rho'(g_2)(\eta e_\alpha) - \eta e_\alpha) \cdot j_\Psi(g_2) =  \\ =
\psi(\rho'(g_1)(\eta e_\alpha) - \eta e_\alpha) \cdot \psi(\rho'(g_1\cdot g_2)(\eta e_\alpha) - \rho'(g_1)(\eta e_\alpha)) \cdot j_\Psi (g_1) \cdot j_\Psi(g_2) = c^\Psi_\alpha(\eta)(g_1\cdot g_2). \qedhere \end{multline} \end{proof}
 
\begin{lemma}\label{glue_c} Let $\alpha, \beta\in \Phi$ be roots and let $\Psi_0, \Psi_1$ be two subsystems from $A_3(\Phi, \beta)$.
 Set $c_0 := c^{\Psi_0}_\alpha(\eta)$ and $c_1 := c^{\Psi_1}_\alpha(\eta)$.
 Then $c_0(z^{\Psi_0}_\beta(s, \xi)) = c_1(z^{\Psi_1}_\beta(s, \xi))$ for $s\in I$, $\xi\in R$. \end{lemma}
\begin{proof}
In the case $\alpha \perp \beta$ the statement of the lemma follows from the definition of $c_\alpha^\Psi(\eta)$ by case analysis (see~\ref{defc1},~\ref{defc2}).
Thus, in what follows we can assume $\alpha\not\perp\beta$.

In the case $\alpha\not\in\Psi_1\cup\Psi_2$ both $c_0$ and $c_1$ are defined by identity~\ref{defc2} 
and the required statement follows from the fact that this definition is based on the relations inside $\G(\Phi, R)$.

Now consider the case $\alpha\in\Psi_1\setminus\Psi_0$. 
We have that $\alpha\neq\pm\beta$, $\alpha\not\perp\Psi_0$ and hence either $\alpha+\beta \in \Phi$ or $\alpha-\beta \in \Phi$, suppose, for example, the former.
Expanding the definition of $c_0$ and $c_1$ and using relation~\eqref{Z2} of Lemma~\ref{Zrels} we get that
\begin{multline}\nonumber c_\alpha^{\Psi_1}(\eta)(z_\beta(s, \xi)) = j_{\Psi_1}(z_\beta(s, \xi)^{x_\alpha(\eta)}) = j_{\Psi_1}(y_{\alpha} (- s\xi \eta) \cdot y_{\alpha+\beta} (N_{\alpha,\beta}\cdot s\eta)) \cdot j_{\Psi_1}(z_{\beta}(s, \xi)) =  \\
= \psi\left(\rho'(z_\beta(s, \xi))(\eta e_\alpha) - \eta e_\alpha\right) \cdot j_{\Psi_0}(z_\beta(s, \xi)) = c_\alpha^{\Psi_0}(\eta)(z_\beta(s, \xi)). \end{multline}
It remains for us to consider the case $\alpha\in\Psi_0\cap \Psi_1$ when both $c_0$ and $c_1$ are defined by the second line of identity~\ref{defc1}.
In this case the required statement follows from identities of Lemma~\ref{Zrels} and Proposition~\ref{SwanP} provided $\beta \neq \alpha$ or if $\xi=0$.
  
Now consider the case when $\alpha = \beta$ and $\Lambda := \Psi_0\cap \Psi_1$ has type $\rA_2$. Choose any parabolic set of roots $S\subseteq \Lambda$ not containing $\beta$.
By the second part of Lemma~\ref{titsLemma} we can rewrite $z_\beta(s, \xi)$ modulo relations of $\St(\Psi_0\cap\Psi_1, R, I)$ as a product of elements of the form $x_\gamma(s_0)$, $\gamma \in \Lambda$, $z_{\delta}(s_1, \xi_1)$, $\delta\in\Sigma_S\subset S$.
By the above argument, the values of $c_0$ and $c_1$ agree on each such factor, hence they agree on $z_\beta(s, \xi)$.
  
In the last remaining case when $\Psi_0\cap\Psi_1$ has type $\rA_1$ or $\rA_1+\rA_1$ we can apply Lemma~\ref{a2pc} and find a subsystem $\Psi$ containing $\beta$ such that
$\Psi_0\cap \Psi$, $\Psi_1\cap \Psi$ have type $\rA_2$. By the previous paragraph we have $c_0(z^{\Psi_0}_\beta(s, \xi))=c^\Psi_\alpha(\eta)(z^{\Psi}_\beta(s, \xi))=c_1(z^{\Psi_1}_\beta(s, \xi))$. \end{proof}

\begin{rem}\label{Stronger} Let $G'$ be a group and let $\{d_\Psi\}_{\Psi\in A_3(\Phi)}$ be a collection of morphisms $\St(\Psi, R, I)\rightarrow G'$ satisfying 
$$d_{\Psi_0}(z^{\Psi_0}_\beta(s, \xi)) = d_{\Psi_1}(z^{\Psi_1}_\beta(s, \xi)),\ \Psi_0, \Psi_1\in A_3(\Phi, \beta).$$
Notice that in this situation we can show that 
$$d_{\Psi_0}\left(\St(i_{\beta, \Psi_0}, R, I)(g)\right) = d_{\Psi_1}\left(\St(i_{\beta, \Psi_1}, R, I) (g)\right)\ \text{for any}\  g\in\St(\langle \{ \beta \} \rangle, R, I).$$

Indeed,if $\Psi_0\cap\Psi_1$ has type $\rA_2$ this follows from the fact that $\St(i_{\alpha, \Psi_1\cap\Psi_0}, R, I)(g)$ can be rewritten as a product of $z_\beta(s, \xi)$, $\beta\in\Psi_0\cap\Psi_1$ by Lemma~\ref{titsLemma}.
Otherwise, we can repeat the same argument as in the last paragraph of the previous lemma.
\end{rem}

From the previous remark, Lemma~\ref{glue_c} and the universal property of $G$ it follows that there exists a unique map $c_\alpha(\eta)\colon G\rightarrow G$ such that $c_\alpha(\eta)j_\Psi = c^\Psi_\alpha(\eta)$. 
 
In order to define the action of $\St(\Phi, R)$ on $G$ it remains to show that $c_\alpha(\eta)$ satisfy Steinberg relations~\ref{Radd}, \ref{Rcf21}, \ref{Rcf22}.
\begin{lemma} The maps $c_\alpha(\eta)$ satisfy the following relations:
\begin{enumerate}
  \item $c_\alpha(\xi) c_\alpha(\eta) c_\alpha(-\xi-\eta) = \mathrm{id}_G,\ \xi,\eta\in R;$
  \item \label{lemrel2} $c_\alpha(\xi) c_\beta(\eta) c_\alpha(-\xi) c_\beta(-\eta)= \mathrm{id}_G,\ \alpha+\beta\notin\Phi,\ \alpha\neq -\beta;$
  \item $c_\alpha(\xi) c_\beta(\eta) c_\alpha(-\xi) c_\beta(-\eta) c_{\alpha+\beta}(-N_{\alpha\beta}\xi \eta) = \mathrm{id}_G,\ \alpha+\beta\in\Phi.$
 \end{enumerate} \end{lemma}
 \begin{proof}
  Denote by $\theta = \prod\limits c_{\alpha_i}(\xi_i)$ any of the maps in the left-hand side of the above relations.
  It suffices to check that $\theta$ fixes $g = z_\gamma(s, \zeta)\in G$ for all $\gamma\in\Phi$, $s\in I$, $\zeta\in R$.
  Set $\Psi := \langle \{ \alpha, \beta, \gamma \} \rangle$. 
  Clearly, $\rk(\Psi)\leq 3$ and we need to analyze several cases.
  \begin{itemize}
   \item Case $\Psi\subseteq \Psi'\in A_3(\Phi)$. In this case the needed equality holds because it holds in $\St(\Psi', R, I)$, indeed
   $\theta(g) = \theta(j_\Psi(g)) = j_\Psi (g^{\prod\limits x_{\alpha_i}(\xi_i)}) = j_\Psi(g) = 1$ as required.
   \item Case $\Psi = 3\rA_1$. In this case both $c_\alpha(\pm \xi)$ and $c_\beta(\pm \eta)$ fix $g$ by definition.
   \item Case $\Psi = \rA_2 + \rA_1$. If $\gamma\perp\langle \{ \alpha,\beta \} \rangle$ then, as before, $g$ is fixed by all factors of $\theta$.
   If $\alpha\perp\langle \beta, \gamma \rangle$ (only possible in the case of relation~\eqref{lemrel2}) then $c_\alpha(\pm \xi)$ act identically 
   on $z_\delta(a, \zeta)$, $\delta\in\langle \{ \beta, \gamma \} \rangle$, hence $\theta (g) = c_\beta(\eta) c_\beta(-\eta) (g) = g$.
   Case $\beta\perp\langle \{ \alpha, \gamma \} \rangle$ is similar to the one just considered.
  \end{itemize} \end{proof}

\begin{proof}[Proof of Theorem~\ref{relPres}]
From the above lemmata it follows that $G$ is a $\St(\Phi, R)$-group and the map $i\colon G\rightarrow \St(\Phi, R, I)$ is $\St(\Phi, R)$-equivariant.
In order to construct the inverse map we use the presentation of $\St(\Phi, R, I)$ from the statement of Proposition~\ref{SwanP}.
Indeed, every relation from Proposition~\ref{SwanP} holds in $G$ since it occurs among the defining relations of some factor $\St(\Psi, R, I)$ of $G$, $\Psi\in A_3(\Phi)$. \end{proof}

\section {Proof of main results} \label {mainReduction}
Throughout the current section we consider the relative Steinberg group $\St(\Phi, R[t], tR[t])$ as a subgroup of $\St(\Phi, R[t])$ thanks to Lemma~\ref{relSubgr}.

\subsection{Glueing Tulenbaev maps}
The main goal of this subsection is to demonstrate an analogue of Lemma~\ref{Tmap} for $\Phi = \rE_6,\rE_7,\rE_8$.
The idea of the proof is to construct the map $T_\Phi$ as in~\ref{diagT} by glueing maps $T_\Psi$ for $\Psi\in A_4(\Phi)$.
First of all, we note that $\rA_3$'s in the statement of Theorem~\ref{relPres} can be replaced with $\rA_4$'s.
\begin{cor}\label{relPresE} For $\Phi=\rE_\ell$, $\ell=6,7,8$ there exists an isomorphism 
$$\St(\Phi, R, I) \cong \varinjlim_{\mathcal{G}_{1,4}}\St(-, R, I).$$ \end{cor}
\begin{proof} This is a formal corollary of Theorem~\ref{relPres}, the first statement of Lemma~\ref{ext34} and the universal property of colimits. 
\end{proof}

Denote by $i_\Psi$ the map $\St(\Psi, R\ltimes tR_a[t])\rightarrow\St(\Phi, R\ltimes tR_a[t])$ induced by embedding $\Psi\hookrightarrow\Phi$ of root systems.
The next step is to show that the maps $i_\Psi\circ T_\Psi$ agree on $z_\gamma(s, \xi)$.
\begin{lemma}\label{glue_t} The value of $i_\Psi\circ T_\Psi$ on $z_\gamma(s, \xi)$, $s \in tR_a[t]$, $\xi\in R_a[t]$ does not depend on the choice of $\Psi\in A_4(\Phi, \gamma)$. \end{lemma}
\begin{proof}
From the construction of Tulenbaev's map $T$ we get that $T(X_{e_i, s e_j})=X_{e_i, s e_j}$, for any $s\in tR_a[t]$, $1\leq i\neq j\leq 5$, see~\cite[Lemma~2.3]{T}.
This implies that $T$ preserves the elements $z_\gamma(s, 0) = y_\gamma(s)$.

Now let $\xi'$ be an element of $R[t]$.
From the commutativity of diagram~\ref{diagT} it follows that $T_\Psi(z_\gamma(s, \lambda_a(\xi'))) = z_\gamma(s, \theta(\xi'))$. 
This shows that the statement of the lemma holds for $\xi \in \Img(\lambda_a)$.

Now consider the case $\xi\not\in \Img(\lambda_a)$. Let $\Psi_0, \Psi_1 \in A_4(\Phi, \gamma)$ be two subsystems.
By the second part of Lemma~\ref{ext34} we may restrict ourselves to the case when $\Psi_0\cap\Psi_1$ contains a subsystem of type $\rA_2$ containing $\gamma$.
We can write $\gamma = \alpha+\beta$ for $\alpha,\beta \in\Psi_0\cap\Psi_1$ such that $N_{\beta,\alpha} = 1$.
We can also choose $n$ such that $a^n \xi \in \Img(\lambda_a)$. Set $b := a^n$, $u := a^{-n}s$.
Direct computation with relations from Proposition~\ref{SwanP} and Lemma~\ref{Zrels} shows that
\begin{multline}\nonumber z_{\gamma}(ub, \xi) = [y_{\beta}(u) ^ {x_{-\gamma}(\xi)}, y_{\alpha}(b) ^{x_{-\gamma}(\xi)}] = 
                 y_{-\alpha} (-u\xi)\cdot y_{\beta}(u)\cdot y_{\alpha}(ub^2\xi)\cdot y_{\gamma}(ub)\cdot \\
                 z_{\beta}(-u, -b\xi)\cdot y_{-\beta}(- ub^2\xi^2)\cdot y_{-\gamma}(-ub\xi^2)\cdot z_{-\alpha}(u\xi, -b). \end{multline}
Consequently, one can rewrite $z_\gamma(s, \xi)=z_\gamma(ub, \xi)$ as a product of $z_{\gamma'}(s', \lambda_a(\xi'))$.
By the first part of the argument, functions $T_{\Psi_0}$ and $T_{\Psi_1}$ coincide on each such factor, hence they are identical on $z_\gamma(s, \xi)$. \end{proof}

Let $A$ be a ring, $B$ an $A$ algebra and $b\in B$.
Denote by $\ev{A}{B}{b}\colon A[t]\rightarrow B$ the morphism of $A$-algebras evaluating each polynomial $p(t)\in A[t]$ at $b$, i.e. $\ev{A}{B}{b} (p(t)) = p(b)$.

\begin{lemma}[Dilation principle]\label{QScor} Let $h\in\St(\Phi, R[t], t R[t])$ be such that $\lambda_a^*(h) = 1$ in $\St(\Phi, R_a[t])$. 
Then for sufficiently large $n$ one has $\ev{R}{R[t]}{a^n\cdot t}^*(h) = 1.$ \end{lemma}
\begin{proof} From the universal property of colimits and the previous lemma (cf. also Remark~\ref{Stronger}) it follows that there exists a unique $T_\Phi$ such that $T_\Phi \circ j_\Psi = i_\Psi \circ T_\Psi$ and $\theta^* = T_\Phi \circ \lambda_a^*$.
As before, $\theta$ denotes the map $R[t]\rightarrow R\ltimes tR_a[t]$ that localizes at $a$ coefficients of terms of degree $\geq 1$ (cf.~Lemma~\ref{Tmap}).

Let $(A_i, f_{ij})$ be the directed system of rings defined by
$$A_i:=R[t],\ f_{ij} := \ev{R}{R[t]}{a^{j-i} \cdot t},\ 0 \leq i\leq j.$$
It is easy to check that $\varinjlim A_i$ coincides with $R\ltimes tR_a[t]$.
The morphisms $A_j\rightarrow \varinjlim_i A_i \cong R\ltimes tR_a[t]$ are exactly $ \ev{R}{R\ltimes tR_a[t]}{a^{-j} \cdot t}$.
Steinberg group functor commutes with colimits over directed systems (cf.~\cite[Lemma~2.2]{T}), therefore
$$\St(\Phi, R\ltimes tR_a[t]) = \St(\Phi, \varinjlim_i A_i) \cong \varinjlim_i \St(\Phi,A_i).$$
By hypothesis, $\theta^*(h)=T_\Phi(1)=1$ and it remains to use the definition of colimit. \end{proof}

\subsection{Quillen---Suslin local-global principle and centrality of $\K_2$}
\begin{lemma}\label{T23}
 Let $\Phi=\rE_\ell$, $\ell=6,7,8$.
 Let $a$, $b$ be elements of $R$ which span $R$ as an ideal and let $\alpha\in\St(\Phi, R[t], t R[t])$ be such that $\lambda_a^*(\alpha)=1$, $\lambda_b^*(\alpha)=1$.
 Then $\alpha=1$.
\end{lemma}
\begin{proof}
 We reproduce the argument of \cite[Lemma~2.5]{T}. 
 Set $S := R[t, t_1]$.
 Consider the following element of $\St(\Phi, S[t_2])$:
 $$\beta(t, t_1, t_2) := \alpha(t_1t) \cdot  \alpha^{-1}((t_1+t_2) t) = \ev{R}{S[t_2]}{t_1t}^*(\alpha) \cdot \ev{R}{S[t_2]}{(t_1 + t_2)t}^*(\alpha^{-1}).$$
 It is easy to see that $\beta(t, t_1, t_2)$ belongs to $\Ker(\eval{t_2}{S}{S}{0}^*\colon\St(\Phi, S[t_2]) \rightarrow \St(\Phi, S))$ and hence by Lemma~\ref{relSubgr} lies in $\St(\Phi, S[t_2], t_2 S[t_2])$. 
 On the other hand, \begin{multline}\nonumber 
 \lambda_{a}^*(\beta(t, t_1, t_2)) = \left(\lambda_a\circ \ev{R}{S[t_2]}{t_1t}\right)^*(\alpha) \cdot \left(\lambda_a\circ \ev{R}{S[t_2]}{(t_1 + t_2)t}\right)^*(\alpha^{-1}) = \\
  = \ev{R_a}{S_a[t_2]}{t_1t}^*(\lambda_{a}^*(\alpha)) \cdot \ev{R_a}{S_a[t_2]}{(t_1 + t_2)t}^*(\lambda_{a}^*(\alpha^{-1})) = 1 \end{multline}
 Similarly, $\lambda_{b}^*(\beta(t, t_1, t_2)) = 1$.
 In view of Lemma~\ref{QScor}, there exists $n$ such that $\beta(t, t_1, a^n t_2) = 1$, $\beta (t, t_1, b^n t_2) = 1$.
 By assumption, there exist $r, s\ \in R$ such that $r a^n + s b^n = 1$.
 Finally, $$1 = \beta(t, 1, -sb^n)\beta(t, ra^n, -ra^n) = \alpha(t)\cdot \alpha^{-1}(ra^n\cdot t) \cdot \alpha(ra^n\cdot t) \cdot \alpha^{-1}(0) = \alpha(t). \qedhere$$ \end{proof}  
  
\begin{proof}[Proof of Theorem~\ref{QSlemma}]
 It suffices to show ``if'' part of the statement. 
 Consider the set $Q(\alpha)$ consisting of elements $a\in R$ such that $\lambda_a^*(\alpha)$ is trivial.
 In the literature the set $Q(\alpha)$ is usually called the {\it Quillen set} of $\alpha$.
 It is also assumed that $Q(\alpha)$ contains the ideal of nilpotent elements of $R$.
 
 Let us show that $Q(\alpha)$ is an ideal of $R$.
 Indeed, let $c$ be a nonnilpotent element lying in the ideal $\langle a, b\rangle$ spanned by $a,b \in Q(\alpha)$.
 We want to show that $c\in Q(\alpha)$.
 Passing to the ring $R_c$ we see that $\lambda_c(a)$ and $\lambda_c(b)$ together generate $R_c$ as an ideal.
 Without loss of generality we may assume that both $\lambda_c(a)$ and $\lambda_c(b)$ are nonnilpotent.
 
 Clearly, $\lambda^*_{\lambda_c(a)}(\lambda_c^*(\alpha)) = \lambda^*_{\lambda_a(c)}(\lambda_a^*(\alpha)) = 1$ and 
 $\lambda^*_{\lambda_c(b)}(\lambda_c^*(\alpha)) = \lambda^*_{\lambda_b(c)}(\lambda_b^*(\alpha)) = 1$.
 Applying Lemma~\ref{T23} to $\lambda_c^*(\alpha)\in\St(\Phi, R_c, tR_c[t])$ we obtain that $\lambda_c^*(\alpha) = 1$ hence $c\in Q(\alpha)$. 
 
 If $Q(\alpha)$ is a proper ideal, then it is contained in some maximal ideal $M$ and there exists $s\in R \setminus M$
 such that $\lambda_s^*(\alpha)=1$ contrary to the definition of $Q(\alpha)$ and $M$.
 This shows that $Q(\alpha)=R$ and $\alpha=1$.
\end{proof}

\begin{lemma}\label{stability} Let $\Phi$ be a simply laced root system of rank $\geq 2$.
Then there exist a vertex $\alpha_1$ on $D(\Phi)$ such that for every local ring $R$ the map $\K_2(\langle \{ \alpha_1 \} \rangle, R) \rightarrow \K_2(\Phi, R)$ is surjective. \end{lemma}
\begin{proof} The result immediately follows from the surjective stability theorem for $\K_2$ (see \cite[Corollary~3.2]{St78}, \cite[Theorem~4.1]{St78}). \end{proof}

\begin{proof}[Proof of Theorem~\ref{centrality}]
 Let $\alpha$ be a simple root and $R$ be arbitrary commutative ring.
 Consider the following subgroups of $\St(\Phi, R)$:
 $$U^+_{\alpha}(R) := \langle X_\beta(R) \mid \beta\in \Sigma_\alpha\rangle,\ U^-_{\alpha}(R) := \langle X_\beta(R) \mid \beta\in -\Sigma_\alpha\rangle.$$
 Using commutator identities~\ref{Rcf21}--\ref{Rcf22} it is easy to show that $U^\pm_{\alpha}(R)$ are normalized by $\Img(\St(\Delta_\alpha, R)\rightarrow \St(\Phi, R))$.
 Recall also that the restriction of the canonical map $\phi$ from~\ref{basicExSeq} to either of $U^\pm_{\alpha}(R)$ is an isomorphism.
 Together with the previous statement this implies that $U^\pm_{\alpha}(R)$ are centralized by $\Img(\K_2(\Delta_\alpha, R)\rightarrow \St(\Phi, R))$.

 Now let $\alpha_1$ be the simple root from the statement of the previous lemma.
 Take any other simple root $\alpha\neq \alpha_1$ (so that $\alpha_1 \in \Delta_\alpha$).
 The subgroups $U_\alpha^\pm(R)$ together generate $\St(\Phi, R)$ as an abstract group (cf.~\cite[Lemma~2.1]{S}),
 therefore it suffices to show that $h := y_\beta(t)^g \cdot y_{\beta}(-t)\in\St(\Phi, R[t], tR[t])$ is trivial for $\beta\in\pm \Sigma_\alpha$, $g\in \Img(\K_2(\Phi, R)\rightarrow\K_2(\Phi, R[t]))$
 (we can specialize $t$ to any element of $R$).
 
 By Theorem~\ref{QSlemma} we are left to prove that $\lambda_M^*(h) = 1$ for every maximal ideal $M$ of $R$. 
 But by the previous lemma $\lambda_M^*(g)$ lies in the image of $\K_2(\Delta_\alpha, R_M)\rightarrow \K_2(\Phi, R_M[t])$
 and hence by the first part of the proof centralizes $U_\alpha^\pm(R_M[t])$.
\end{proof}

\begin{proof}[Proof of Corollary~\ref{topoCorr}]
 Statement~\ref{UCE} is a consequence of the fact that under the assumptions of Theorem~\ref{centrality} the group $\St(\Phi, R)$ is superperfect,
 i.e. its first two homology groups are trivial (see~\cite[Theorem~5.3]{St1}).   
 
 To see that $\St(\Phi, R, I)$ is a relative central extension of $\pi^*\colon \St(\Phi, R)\rightarrow \St(\Phi, R/I)$ recall from~\ref{relativeSteinbergDef}
 that $$G_1 \cap G_2 \subseteq \K_2(\Phi, D(R, I)) \subseteq \Cent(\St(\Phi, D(R, I))),$$ which implies that the induced action of $\St(\Phi, R/I)$ on $(G_1\cap G_2)/C$ is trivial.
 To see that $\St(\Phi, R, I)$ is a {\it universal} relative central extension use the sufficient condition of~\cite[Proposition~6]{Lo} (compare also with the proof of~\cite[Proposition~8]{Lo}).
 
 \end{proof} 
 
\begin{proof}[Proof of Corollary~\ref{topoCorr2}] 
 Statement~\ref{RelSeq} can be obtained by repeating the proof of~\cite[Theorem~4]{Lo}.
 It is obvious that $\pi_1(X(R))=\K_1(\Phi, R)$. 
 Repeat the argument of~\cite[Corollary~3]{La} to show that $\pi_2(X(R))=\K_2(\Phi, R)$.
 Finally, Gersten formula~\ref{Gersten} can be obtained by repeating the proof of the main theorem of~\cite{Ge}. \end{proof}

\printbibliography

\end{document}